\documentclass[reqno,11pt]{amsart}

\usepackage{amsmath}
\usepackage{amscd}
\usepackage{amssymb}
\usepackage{enumerate}
\usepackage{amsfonts}
\usepackage{graphicx}
\usepackage[all]{xy}
\usepackage{mathrsfs}
\usepackage{fullpage}
\usepackage{hyperref}
\usepackage{verbatim} 
\usepackage{bbm}

\numberwithin{equation}{section}

\newcommand{\rar}[1]{\stackrel{#1}{\longrightarrow}}

\newcommand{\isom}{\rar{\simeq}}

\newcommand{\into}{\hookrightarrow}
\newcommand{\onto}{\twoheadrightarrow}

\newcommand{\Ga}{\Gamma}

\newcommand{\la}{\lambda}

\newcommand{\bC}{{\mathbb C}}

\newcommand{\bG}{{\mathbb G}}

\newcommand{\bM}{{\mathbb M}}

\newcommand{\bV}{{\mathbb V}}
\newcommand{\bW}{{\mathbb W}}

\newcommand{\bZ}{{\mathbb Z}}

\newcommand{\cD}{{\mathcal D}}

\newcommand{\cP}{{\mathcal P}}
\newcommand{\cQ}{{\mathcal Q}}

\newcommand{\cV}{{\mathcal V}}

\newcommand{\sC}{{\mathscr C}}
\newcommand{\sD}{{\mathscr D}}

\newcommand{\sL}{{\mathscr L}}

\newcommand{\sO}{{\mathscr O}}

\newcommand{\fD}{{\mathfrak D}}

\newcommand{\fF}{{\mathfrak F}}

\newcommand{\fZ}{{\mathfrak Z}}

\newcommand{\fb}{{\mathfrak b}}

\newcommand{\fg}{{\mathfrak g}}

\newcommand{\fn}{{\mathfrak n}}

\newcommand{\fp}{{\mathfrak p}}

\newcommand{\fz}{{\mathfrak z}}

\newcommand{\fgh}{\hat{\fg}}

\newcommand{\on}{\operatorname}

\newcommand{\Ker}{\on{Ker}}

\newcommand{\End}{\on{End}}
\newcommand{\Hom}{\on{Hom}}
\newcommand{\Ext}{\on{Ext}}
\newcommand{\Aut}{\on{Aut}}
\newcommand{\Spec}{\on{Spec}}

\newcommand{\Ind}{\on{Ind}}

\newcommand{\Rep}{\on{Rep}}
\newcommand{\gr}{\on{gr}}

\newtheorem{thm}{Theorem}[section]
\newtheorem{cor}[thm]{Corollary}
\newtheorem{lem}[thm]{Lemma}

\newtheorem{prop}[thm]{Proposition}
\newtheorem{prop-const}[thm]{Proposition-Construction}

\newtheorem{fthm}[thm]{Finiteness Theorem}

\theoremstyle{remark}
\newtheorem{rem}[thm]{Remark}
\newtheorem{rems}[thm]{Remarks}

\newcommand{\Tor}{\on{Tor}}

\newcommand{\Fun}{\on{Fun}}

\newcommand{\Lie}{\on{Lie}}

\newcommand{\lto}{\rar{}}
\newcommand{\vph}{\varphi}

\newcommand{\RHom}{\on{RHom}}

\newcommand{\Res}{\on{Res}}
\newcommand{\Gr}{\on{Gr}}

\newcommand{\Op}{\on{Op}}

\newcommand{\modules}{\on{-mod}}

\newcommand{\Adjoint}[4]{\xymatrix@1{#2 \ar@<.4ex>[r]^-{#1} & #3 \ar@<.4ex>[l]^-{#4}}}

\newcommand{\ld}{\check}

\title{A geometric proof of the Feigin-Frenkel theorem}

\author{Sam Raskin}

\date{\today}

\begin{document}

\begin{abstract}
We reprove the theorem of Feigin and Frenkel relating the center of the critical level enveloping algebra of the Kac-Moody algebra for a semisimple
Lie algebra to opers (which are certain de Rham local systems with extra structure) for the Langlands dual group. Our proof incorporates
a construction of Beilinson and Drinfeld relating the Feigin-Frenkel isomorphism to (more classical) Langlands duality through
the geometric Satake theorem.
\end{abstract}

\maketitle

\setcounter{tocdepth}{1}
\tableofcontents

\section{Introduction}\label{s:introduction}

\subsection{} The goal for this article is to give a second proof of the Feigin-Frenkel theorem that two polynomial algebras on countably many generators are isomorphic (subject to certain compatibilities: see Theorem \ref{t:ff}).

\subsection{} Recall that attached to an invariant symmetric bilinear form $\kappa$ on a semisimple Lie algebra $\fg$, we have the 
``(affine) Kac-Moody algebra $\fgh_{\kappa}$ of level $\kappa$" constructed as follows.

Let $\fg((t)):=\fg{\otimes}\bC((t))$ denote the Lie algebra from $\fg$ obtained from extension of scalars to the field $K=\bC((t))$ of Laurent series. We consider $\fg((t))$ as a topological Lie algebra over $\bC$ (in the sense of
\cite{top-algs}). Then $\fgh_{\kappa}$ is the central extension of $\fg((t))$ by the trivial one-dimensional Lie algebra $\bC$ defined by the 2-cocycle $(f,g)\mapsto \Res(\kappa(f,dg))$.

We have the abelian category $\fgh_{\kappa}\modules$ of discrete modules over $\fgh_{\kappa}$ such that the element $1\in\bC\subset\fgh_{\kappa}$ acts by the identity operator.

\subsection{} Note that the tautological map $\fg((t))\into\fgh_{\kappa}$ of topological vector spaces realizes $\fg[[t]]:=\fg{\otimes}\bC[[t]]$ as a Lie subalgebra of $\fgh_{\kappa}$ (since
the cocycle vanishes on this algebra). Therefore, we can induce modules from $\fg[[t]]$ to obtain objects of $\fgh_{\kappa}$. We define the vacuum module $\bV_{\kappa}\in\fgh_{\kappa}\modules$ as the module obtained by
inducing the trivial representation from $\fg[[t]]$ to $\fgh_{\kappa}$.

\subsection{} In this article, we are concerned with the case of ``critical level," i.e. $\kappa=\kappa_{crit}$ the invariant bilinear form given as $-\frac{1}{2}$ times the Killing form for $\fg$. We let
$\fgh$ denote $\fgh_{\kappa_{crit}}$ and $\bV$ denote $\bV_{\kappa_{crit}}$ the critical level vacuum module.

The Feigin-Frenkel theorem \cite{feigin-frenkel} describes the algebra $\fz:=\End_{\fgh\modules}(\bV)$ in terms of ``opers" for the Langlands dual group, whose definition we will presently recall. 
But before proceeding, let us record a first ``miracle": the center $\fZ$ of the category $\fgh\modules$ surjects onto $\fz$, so in particular $\fz$ is commutative (see \cite{bd} 2.9).

\subsection{} Let $G$ be the simply-connected semisimple group with Lie algebra $\fg$, and let us fix once and for all a Borel $B$ of $G$. Let $\ld{G}$ be the
Langlands dual group (i.e. the adjoint group for the dual root system), which is equipped with the dual Borel $\ld{B}$. We denote
the associated Lie algebras by $\fb$,$\ld{\fg}$ and $\ld{\fb}$ respectively.

An \emph{oper on the formal disc $\cD:=\Spec(\bC[[t]])$} is a pair $(\cP_{\ld{B}},\nabla=\nabla_{\ld{G}})$ consisting
of a $\ld{B}$-bundle $\cP_{\ld{B}}$ on $\cD$ with a connection $\nabla$ on the associated
$\ld{G}$-bundle $\cP_{\ld{G}}$, and which satisfies the following non-degeneracy property.

There is an obstruction $c(\nabla)\in(\ld{\fg}/\ld{\fb})_{\cP_{\ld{B}}}\otimes_{\sO_{\cD}}\omega_{\cD}$ to 
the preservation of $\nabla$ under the reduction to $\ld{B}$ (e.g., trivialize the $\ld{B}$-bundle and take the equivalence
class of the corresponding connection form). The oper condition for $(\cP_{\ld{B}},\nabla)$ is that:
\begin{enumerate}

\item\label{i:grif} $c(\nabla)$ lies in the sub-bundle of 
$(\ld{\fg}/\ld{\fb})_{\cP_{\ld{B}}}\otimes_{\sO_{\cD}}\omega_{\cD}$ associated to the $\ld{\fb}$-submodule
of $\ld{\fg}/\ld{\fb}$ which is the span of the weight spaces $(\ld{\fg}/\ld{\fb})_{\cP_{\ld{B}}}^{\check{\alpha}}$
for the negative simple roots $\check{\alpha}$ of $\ld{\fg}$.

\item\label{i:nondeg} By (\ref{i:grif}), $c(\nabla)$ lies in the direct sum of the line bundles $(\ld{\fg}/\ld{\fb})_{\cP_{\ld{B}}}^{\check{\alpha}}\otimes_{\sO_{\cD}}\omega_{\cD}$ where
$\check{\alpha}$ ranges over negative simple roots of $\ld{\fg}$. We require that the corresponding sections
are everywhere non-vanishing, i.e., that  they trivialize the line bundles $(\ld{\fg}/\ld{\fb})_{\cP_{\ld{B}}}^{\check{\alpha}}\otimes_{\sO_{\cD}}\omega_{\cD}$.

\end{enumerate}

\begin{rems}

\begin{enumerate}

\item Condition \ref{i:grif} for $GL_n$ is Griffiths' transversality condition.

\item Condition \ref{i:nondeg} implies that the connection $\nabla$ is \emph{never} preserved by the reduction to the Borel.

\item This definition seems to have first appeared in the $GL_n$-case in \cite{commutative-subrings}, and for general $G$
in \cite{drinfeld-sokolov}.

\end{enumerate}

\end{rems}

The above definition naturally extends to a moduli problem,\footnote{For $A$ a $\bC$-algebra, one should give a $\ld{B}$-bundle
on $\Spec(A[[t]])$, and on the associated $\ld{G}$-bundle a connection along the $t$-coordinate satisfying conditions similar to the above.}
and there is an associated scheme (non-canonically isomorphic to an infinite-dimensional affine space) which we denote by
$\Op_{\ld{G}}(\cD)$.

\begin{rem}\label{r:bg}
Note that $\Op_{\ld{G}}(\cD)$ has a canonical $\ld{B}$-bundle whose fiber at an oper $(\cP_{\ld{B}},\nabla_{\ld{G}})$ is the fiber of the $\ld{B}$-bundle $\cP_{\ld{B}}$ at $0$.
In particular, we also have the induced $\ld{G}$-bundle on $\Op_{\ld{G}}(\cD)$.
\end{rem}

\subsection{} The Feigin-Frenkel theorem from \cite{feigin-frenkel} gives an isomorphism $\vph:\Spec(\fz)\isom \Op_{\ld{G}}(\cD)$.
Let us recall two compatibilities that this isomorphism satisfies.

\subsection{} Let $\Aut$ be the group scheme (of infinite type) of automorphisms of $\cD$ which preserve the closed point, and
let $\Aut^+$ be the group ind-scheme of all automorphisms of $\cD$. I.e., for a commutative $\bC$-algebra $A$, the $A$-points of $\Aut$ 
are power series without constant term and with invertible $t$-coefficient, while $A$-points of $\Aut^+$ allow nilpotent constant term. The
group law (in both cases) is defined by composition of power series. Note that the action of $\Aut^+$ on $\cD$ extends to an 
action on the formal punctured disc $\cD^{\times}:=\Spec(\bC((t)))$.

We have a natural action of $\Aut^+$ on the algebra $\fz$. This action is induced by an action on $\bV$, which in turn comes from actions on $\fg((t))$ and $\fg[[t]]$.
There is also an action of $\Aut^+$ on $\Op_{\ld{G}}(\cD)$, which comes from the action of $\Aut^+$ on $\cD$.

The first compatibility of the Feigin-Frenkel isomorphism $\vph$ is that it is equivariant for these two $\Aut^+$-actions.

\subsection{} The second compatibility which we recall is somewhat more subtle, and is a principal result of \cite{bd}. To formulate it, we briefly recall the statement
of the geometric Satake equivalence, proved e.g. in \cite{mirkovic-vilonen}.

\subsection{} Let $\Gr_G$ be the affine Grassmannian, i.e., the moduli space of $G$-bundles on $\cD$ with trivialization on $\cD^{\times}$, and let $1_{\Gr_G}$ be its ``trivial" $\bC$-point. Let $G(K)$ be the
group ind-scheme of maps of $\cD^{\times}$ into $G$, and let $G(O)$ be its group subscheme consisting of maps which extend to $\cD$. Note that the associated (Tate) Lie algebras
are $\fg((t))$ and $\fg[[t]]$ respectively. There is an action of $G(K)$ on $\Gr_G$ by changing the trivialization. Via the point $1_{\Gr_G}$, this action realizes $\Gr_G$ as the 
quotient $G(K)/G(O)$ (see \cite{bd} 4.5 for a precise formulation of this result).

Recall from \cite{bd} Section 4 that there is a unique\footnote{Here it is important that $G$ is simply-connected: otherwise there is a neutral gerbe of choices.} central group ind-scheme extension $\widetilde{G(K)}$ 
of $G(K)$ by $\bG_m$ split over $G(O)$ with an isomorphism $\Lie(\widetilde{G(K)})\isom\fgh$ as extensions of $\fg((t))$ by $\bC=\Lie(\bG_m)$ split over $\fg[[t]]$.
Therefore, there is an induced line bundle on $\Gr_G$ which we call ``critical," and we let $D_{\Gr_G}\modules$ be the category of 
critically twisted $D$-modules on $\Gr_G$.\footnote{This category is canonically equivalent to the category of $D$-modules on $\Gr_G$.}

The following is a rough statement of geometric Satake, which suffices for our purposes.

\begin{thm}\label{t:satake}

The convolution and fusion structures on the abelian category $D_{\Gr_G}\modules^{G(O)}$ of strongly $G(O)$-equivariant (critically twisted) $D$-modules on $\Gr_G$ make this category into a 
symmetric monoidal category. There is a canonical equivalence of symmetric monoidal categories
between $D_{\Gr_G}\modules^{G(O)}$ and $\Rep(\ld{G})$ the category of (possibly infinite-dimensional) representations of $\ld{G}$, considered as a symmetric monoidal category
under tensor product. Under this equivalence, the irreducible representation of highest weight $\check{\lambda}$ corresponds to the intersection cohomology $D$-module on the stratum
$G(O)\cdot\check{\lambda}(t)\cdot 1_{\Gr_G}$ in $\Gr_G$.

\end{thm}

For $V\in\Rep(\ld{G})$, we let $\fF_V$ denote the corresponding twisted $D$-module on $\Gr_G$.

\subsection{} We have a global sections functor $\Ga=\Ga(\Gr_G,-):D_{\Gr_G}\modules\lto\fgh\modules$\footnote{It is important to note that we're taking sections 
as a mere $\sO$-module.} coming from the infinitesimal action of $\fgh$ on the critical line bundle on $\Gr_G$.
In fact, this functor naturally maps to $\fgh\modules_{reg}$ the category
of ``regular" $\fgh$-modules: modules on which the kernel of $\fZ\lto\fz$ acts trivially, i.e., on which $\fZ$ acts via its 
quotient $\fz$.\footnote{Here is a sketch of a proof: note that $\Ga(\Gr_G,\delta_e)$
is the vacuum module $\bV$, so the result is true for this $D$-module. For a general point $g\in\Gr_G$, $\Ga(\Gr_G,\delta_g)$ is $\bV$ twisted via the adjoint action by a lift
of $g$ to $G(K)$, so the result is true here by centrality. But then the map from $\fZ$ to differential operators on $\Gr_G$ fiberwise kills the kernel of $\fZ\lto\fz$, so the same
holds true in general.} Finally, one sees that $\Ga:D_{\Gr_G}\modules^{G(O)}\lto\fgh\modules_{reg}^{G(O)}$, where the right hand side is the category of $G(O)$-integrable modules. 

\begin{thm}\emph{(Feigin-Frenkel-Beilinson-Drinfeld)}\label{t:ff}
There is an isomorphism $\Spec(\fz)\lto \Op_{\ld{G}}(\cD)$ compatible with the action of $\Aut^+$ on both sides and such that the functor
$\Ga:D_{\Gr_G}\modules^{G(O)}\lto\fgh\modules_{reg}^{G(O)}$ is naturally a map of categories over the stack $pt/\ld{G}$.
\end{thm}

More concretely, this second condition means the following: for $\fF_V\in D_{\Gr_G}\modules^{G(O)}$ we have:
\[
\Ga(\Gr_G,\fF_V)\isom \bV\otimes_{\fz}\cV_{\Op_{\ld{G}}(\cD)}
\]
\noindent where $\cV_{\Op_{\ld{G}}(\cD)}$ is the twist by $V$ of the canonical $\ld{G}$-bundle on
$\Op_{\ld{G}}(\cD)$ as in Remark \ref{r:bg}. Moreover, this identification is compatible with tensor products of representations.

\begin{rem}
As explained in \cite{bd}, Theorem \ref{t:ff} in this form and the convolution formalism of \emph{loc. cit.} Section 7 combine to give a 
construction of Hecke eigenmodules for global opers on smooth projective curves.
\end{rem}

\begin{rem}
By \cite{bd} 3.5.9, Theorem \ref{t:ff} characterizes the Feigin-Frenkel isomorphism uniquely.
\end{rem}

\subsection{} The goal for this text is to prove Theorem \ref{t:ff} by methods which differ from the original methods. Note that \cite{feigin-frenkel} originally proved that there is
an $\Aut^+$-equivariant isomorphism, but did not prove the compatibility with geometric Satake. Our method differs from theirs in that this compatibility is built into the proof. The present
technique also does not use the ``screening operators" of \emph{loc. cit.}

\subsection{} To do this, we use the ``birth of opers" construction from \cite{bd} to give a (visibly $\Aut^+$-equivariant) map $\Spec(\fz)\lto\Op_{\ld{G}}(\cD)$. Then a construction of \cite{wakimoto}
automatically implies that this map is an isomorphism. 

As in \cite{bd}, the birth of opers construction relies on a purely representation-theoretic description of the modules 
$\Ga(\Gr_G,\fF_V)$, namely, that these modules are direct sums of copies of the vacuum module and that the higher cohomologies of their
global sections vanish. However, \cite{bd} deduces this description from the Feigin-Frenkel theorem. Our method is to prove this description of global sections directly and then
perform the birth of opers construction.

\subsection{} Let us describe the structure of the current text.

In Section \ref{s:rep-theory}, we recall those parts of the structure of $\fgh\modules_{reg}$ which can be proved without the Feigin-Frenkel theorem. In Section \ref{s:ds}, we recall the properties of the
``Whittaker" semi-infinite cohomology functor. Here we also formulate a finiteness statement for semi-infinite cohomology.
In Section \ref{s:vanishing}, the only section of the present text with any claim to originiality, we prove the above description of the cohomology of global sections of critically twisted spherical $D$-modules
on the affine Grassmannian. In Section \ref{s:opers} we recall the birth of opers construction of \cite{bd} and explain how to deduce Theorem \ref{t:ff} from it. In Section \ref{s:finiteness},
we prove this finiteness statement.

\subsection{Acknowledgements} The author is grateful to Sasha Beilinson for patiently introducing us to this subject and for suggesting this problem to us at Botany Pond in the spring of 2009.
The author is greatly indebted to his graduate advisor Dennis Gaitsgory for many helpful conversations: his ideas were indespensible for the present text.

\section{Representation theory at the critical level with regular central character}\label{s:rep-theory}

\subsection{} We will use this section to collect some known facts about the category $\fgh\modules_{reg}$.

\subsection{} Note that $\bV$ has a PBW filtration compatible with the action of $G(O)$, and therefore $\fz=\bV^{G(O)}$ inherits a canonical filtration. We let
$\bV^{cl}=\gr(\bV)$. We define $\fz^{cl}=\gr(\bV)^{G(O)}$. There is a canonical map $\gr(\fz)\into\fz^{cl}$.

The starting point is the following ``quantization" result:

\begin{thm}\label{t:gr}
The natural map $\gr\fz\into\fz^{cl}$ is an isomorphism.
\end{thm}

This is proved in \cite{frenkel} Section 9.3. We will outline the proof from \emph{loc. cit.} in Sections \ref{ss:quant-start}-\ref{ss:quant-finish}.

\subsection{} Note that Theorem \ref{t:gr} implies in particular that $\fz$ is non-canonically isomorphic to a polynomial algebra on countably many generators. Indeed, because polynomial algebras are free
among commutative algebras, it suffices to see that $\fz^{cl}$ is a polynomial algebra. This is shown in \cite{bd} 2.4. 

We therefore deduce that the conclusion of the following result holds for $\fz$:

\begin{lem}\label{l:res}
Any finitely presented module $N$ over an infinite polynomial algebra $A$ over $\bC$ admits a finite resolution by finite rank free $A$-modules.
\end{lem}

\begin{proof}

Finite presentation allows us to reduce the result to the case of a polynomial algebra on finitely many generators, where it is well-known.

\end{proof}

\subsection{} We have the following result:

\begin{prop}\label{p:flat}
The vacuum module $\bV$ is flat over $\fz$.
\end{prop}

\begin{proof}

In \cite{ef} the corresponding semi-classical statement is proved, i.e., that $\bV^{cl}$ is flat over $\fz^{cl}$. 
Since a filtered module is flat if its associated graded is, by Theorem \ref{t:gr} this semi-classical statement suffices.
\end{proof}

\subsection{} The following theorem is shown in \cite{frenkel-teleman}. Note that they deduce the theorem
from the corresponding semi-classical statement, which is proved in \cite{fgt}.

\begin{thm}\label{t:frenkel-teleman}
The $\fz$-modules $\Ext_{\fgh\modules^{G(O)}}^i(\bV,\bV)$ are flat over $\fz$.
\end{thm}

\subsection{} The Frenkel-Teleman theorem implies the following:

\begin{cor}\label{c:invariants}
For a $\fz$-module $N$, the natural map $N\lto(N\otimes_{\fz}\bV)^{G(O)}$ is an isomorphism.
\end{cor}

\begin{proof}

The following argument is given in \cite{affinegrassmannian} 8.6. There is a natural map:
\[
\RHom_{\fgh\modules^{G(O)}}(\bV,\bV)\underset{\fz}{\overset{L}{\otimes}} N\lto\RHom_{\fgh\modules^{G(O)}}(\bV,N\underset{\fz}{\overset{L}{\otimes}}\bV).
\]
\noindent First, we claim that this map is a quasi-isomorphism. Both sides commute with filtered colimits in the 
$N$-variable because they are computed explicitly via the Chevalley complex. Therefore, we may assume $N$ is finitely presented. 
By Lemma \ref{l:res}, $N$ admits a finite resolution by finite rank free $\fz$-modules. 
This reduces the general statement of this quasi-isomorphism to the case of free modules, where it is obvious.

Next, we compute $H^0$ of both sides of this quasi-isomorphism. By Theorem \ref{t:frenkel-teleman}, $H^0$ of the left hand
side just $N$, while $H^0$ of the right hand side is $(N\otimes_{\fz}\bV)^{G(O)}$ by Proposition \ref{p:flat}. It is direct
to see that the induced isomorphism $N\isom(N\otimes_{\fz}\bV)^{G(O)}$ given by taking $H^0$ coincides with the map from the
statement of the corollary.

\end{proof}

\subsection{} The following statement was proved in \cite{bd} 6.2.1. We repeat its proof here for convenience.

\begin{prop}\label{p:generator}
For $M\in\fgh\modules_{reg}^{G(O)}$, $M^{G(O)}\neq{0}$ if $M\neq{0}$.
\end{prop}

\begin{proof}

First, note that the group $G^{(1)}(O):=\Ker(\on{ev}:G(O)\lto G)$ is pro-unipotent and acts on $M$, so
$M^{G^{(1)}(O)}\neq{0}$. The group $G$ acts on $M^{G^{(1)}(O)}$, so the Casimir operator $C_{crit}$ defined by
the critical inner product on $\fg$ acts on it. We will show that $M^{G^{(1)}(O)}=M^{G(O)}$
by showing that the action of $C_{crit}$ is zero.

To this end, recall e.g. from \cite{bd} 3.6.14. that there exists the Sugawara operator
$S_0\in I\subset \fZ$, which is of the form $C_{crit}+U(\fgh)\cdot t\fg[[t]]$.
This immediately gives the result since $I$ is assumed to annihilate $M$.

\end{proof}

\subsection{} Let us note the following useful corollary to Proposition \ref{p:generator}.

\begin{cor}\label{c:injection}
For $M\neq 0$ a regular $G(O)$-integrable $\fgh$-module, the natural map $M^{G(O)}\otimes_{\fz}\bV\lto M$ is a non-zero injection.
\end{cor}

\begin{proof}

That the map is non-zero is the statement of Proposition \ref{p:generator}.
Consider the kernel $K$ of the map $M^{G(O)}\otimes_{\fz}\bV\lto M$. 
We have: 
\[
K^{G(O)}=\Ker\bigg((M^{G(O)}\underset{\fz}{\otimes}\bV)^{G(O)}\lto M^{G(O)}\bigg).
\]
\noindent By Corollary \ref{c:invariants}, the map $(M^{G(O)}\otimes_{\fz}\bV)^{G(O)}\lto M^{G(O)}$ is
an isomorphism. Therefore $K^{G(O)}=0$, so Proposition \ref{p:generator} gives $K=0$ as desired.

\end{proof}

\begin{rem}
Note that \cite{affinegrassmannian} 8.8 shows that this map is actually an isomorphism, but their proof relies on the
Feigin-Frenkel theorem.
\end{rem}

\subsection{} We will also need the following statement, \cite{bd} 6.2.4.

\begin{thm}\label{t:vac-ext}
The group $\Ext_{\fgh\modules_{reg}}^1(\bV,\bV)=0$. 
\end{thm}

As in \emph{loc. cit.}, this implies the following:

\begin{cor}\label{c:vac-ext}
If $M$ is a $G(O)$-integrable regular $\fgh$-module such that $M^{G(O)}$ is a projective $\fz$-module, then the map $M^{G(O)}\otimes_{\fz}\bV\lto M$ is an isomorphism.
\end{cor}

Actually, we will need a version of this with parameters. For a commutative algebra $A$ equipped with an algebra 
morphism\footnote{The following constructions make sense for any continuous morphism $\fZ\lto A$, but Corollary \ref{c:vac-ext-param} 
would certainly fail.} $\fz\lto A$, let $\fgh\modules_A$ denote the abelian category of regular $\fgh$-modules 
equipped with an action of $A$ by endomorphisms compatible with the action of $\fz$. That is, objects are $\fgh$-modules $M$ equipped with 
a $\fz$-algebra morphism $A\lto\End_{\fgh\modules}(M)$. The notion of morphism of such data is clear. Similarly, we have the corresponding $G(O)$-integrable category. Both of these categories are $A$-linear.

\begin{cor}\label{c:vac-ext-param}
If $A$ is flat over $\fz$ and $M\in\fgh\modules_A^{G(O)}$ is such that $M^{G(O)}$ is a flat $A$-module, then the map
$M^{G(O)}\otimes_{\fz}\bV\lto M$ is an isomorphism.
\end{cor}

\begin{proof}

First, we claim that $\Ext_{\fgh\modules_{reg}}^1(\bV,\bV\otimes_{\fz}M^{G(O)})=0$. Indeed, by \cite{fg2} 7.5.1., $\Ext_{\fgh\modules_{reg}}^1(\bV,-)$ commutes with filtered colimits. Since
$A$ is flat over $\fz$ and $M^{G(O)}$ is flat over $A$, $M^{G(O)}$ is flat over $\fz$.
By Lazard's theorem, the condition that $M^{G(O)}$ is a flat $\fz$-module implies that it is a filtered colimit of (finite rank) free $\fz$-modules, so the vanishing now follows from 
Theorem \ref{t:vac-ext}. 

Now let us prove the result. Suppose for the sake of contradiction that the cokernel $C$ of the map $M^{G(O)}\otimes_{\fz}\bV\lto M$ is non-zero. Because
$C$ is regular and $G(O)$-equivariant, $C^{G(O)}=\Hom_{\fgh\modules}(\bV,C)\neq 0$, so there exists non-zero $\bV\lto C$. By the vanishing of the $\Ext$-group
above, this map lifts to a map $\bV\lto M$. But the map $M^{G(O)}\lto C^{G(O)}$ is zero because applying $G(O)$-invariants to the map $M^{G(O)}\otimes_{\fz}\bV\lto M$ gives an isomorphism
by Corollary \ref{c:invariants}. Therefore, we have a contradiction, implying that $C=0$. By Corollary \ref{c:injection} the map $M^{G(O)}\otimes_{\fz}\bV\lto M$ is an injection, which completes
the proof.
\end{proof}

\subsection{}\label{ss:quant-start} For the remainder of this section, we will outline the proof of Theorem \ref{t:gr} following
\cite{frenkel}. Let $I$ be the Iwahori subgroup of $G(K)$, i.e., $I=\on{ev}^{-1}(B)$ for
$\on{ev}:G(O)\lto G$ the evaluation map.
We have the Verma module $\bM_0:=\Ind_{\Lie(I)}^{\fgh}\bC=\Ind_{\fg[[t]]}^{\fgh}M_0$ where $M_0$ is the usual Verma module of weight 0 for $\fg$, considered as a $\fg[[t]]$-module via
the map $\fg[[t]]\lto\fg$. Note that $\bM_0$ is $I$-integrable.

There is a canonical surjection $\vph:\bM_0\lto\bV$ induced by the map $M_0\lto \bC$. Note also that there is a canonical PBW filtration on $M_0$ that is compatible with the action of Iwahori and the map $\vph$.
Therefore, we also have a map on the associated graded level: $\vph^{cl}:\bM_0^{cl}\lto\bV^{cl}$.

\subsection{} Note that $\Aut$ acts on $\bM_0$ and $\bV$ (in fact, all of $\Aut^+$ acts on $\bV$) compatibly with its action on $\fgh$, so each has an operator $L_0:=-t\cdot\frac{\partial}{\partial t}\in \Lie(\Aut)$ 
acting on them diagonalizably with integral eigenvalues (here we have chosen a uniformizer $t$ on $\cD$). The map $\bM_0\lto\bV$ is compatible with the action of $\Aut$ and therefore intertwines the two $L_0$ operators.

\subsection{} The proof of Theorem \ref{t:gr} procedes by first showing the following two results:

\begin{enumerate}

\item The map $\vph^{cl,I}:\bM_0^{cl,I}\lto\bV^{cl,I}$ of $I$-invariant vectors is surjective.

\item The formal character for the action of $L_0$ on $\bM_0^{I}$ exists (i.e., the eigenspaces are finite-dimensional)
and coincides with the $L_0$-character of $\bM_0^{cl,I}$.

\end{enumerate}

The first result admits a purely algebro-geometric interpretation and is proved as Corollary 9.5 in \cite{frenkel}. The second result for $\bM_0^{cl,I}$ follows by explicit computation.
To compute the $L_0$-character of $\bM_0^I$, one notes that $\bM_0^{I}=\End_{\fgh\modules}(\bM_0)$ and then uses the description of $\bM_0$ in terms of Wakimoto modules (\cite{fg2} 13.3.1) and
the computation of endomorphisms of Wakimoto modules (\cite{fg2} 13.1.2).

\subsection{}\label{ss:quant-finish} Now let us show that these two results imply Theorem \ref{t:gr}.

We have the following diagram, in which all maps are compatible with the action of the operator $L_0$:
\[\xymatrix{
\gr(\bM_0^{I})\ar[r]\ar@{^(->}[d]^{\alpha} & \gr(\bV^{I}) \ar@{^(->}[d] & \gr(\bV^{\fg[[t]]})\ar@{_(->}[l]^{\gamma'}\ar@{=}[r]\ar@{^(->}[d] & \gr(\fz)\ar@{^(->}[d] \\
\bM_0^{cl,I}\ar[r]^{\beta} & \bV^{cl,I} & \bV^{cl,\fg[[t]]}\ar@{=}[r]\ar@{_(->}[l]^{\gamma} & \fz^{cl}.
}\]

By computation of $L_0$-characters, we see that $\alpha$ is an isomorphism. Furthermore, we know that $\beta$ is a surjection. Next, we claim that
$\gamma$ and $\gamma'$ are isomorphisms. Indeed, because $\bV^{cl}$ and $\bV$ are $G(O)$-integrable, any vector invariant for the Iwahori subgroup is invariant for all of 
$G(O)$.\footnote{Proof: consider the residual action of $G$ on $\bV^{cl,G^{(1)}(O)}$.} Therefore, we have the following diagram:
\[\xymatrix{
\gr(\bM_0^{I})\ar[r]\ar[d]^{\simeq} & \gr(\bV^{I}) \ar@{^(->}[d] & \gr(\fz)\ar@{^(->}[d]\ar[l]_{\simeq} \\
\bM_0^{cl,I}\ar@{->>}[r] & \bV^{cl,I} & \fz^{cl}\ar[l]_{\simeq}.
}\]
\noindent The left square immediately implies that $\gr(\bV^{I})\into \bV^{cl,I}$ is surjective and therefore an isomorphism, and
this gives the desired result for $\gr(\fz)\into\fz^{cl}$.

\section{Drinfeld-Sokolov functor}\label{s:ds}

\subsection{} In this section, we review the properties of the Drinfeld-Sokolov functor $\Psi$ and formulate the Finiteness Theorem in terms of it.

\subsection{} Fix a character $\chi$ of $\fn((t))$ which vanishes on $\fn[[t]]$ and such that for each positive simple root $e_{\alpha}$ of $\fg$,
we have $\chi(\frac{1}{t}e_{\alpha})\neq 0$.

\subsection{} Suppose that $M$ is a $\fgh$-module. Consider the complex:
\[
\Psi(M):=C(\fn((t)),\fn[[t]],M\otimes\chi)
\]
\noindent of semi-infinite chains of the $\fn((t))$-module $M\otimes\chi$ with respect to the lattice $\fn[[t]]$. This complex is called the
\emph{Drinfeld-Sokolov reduction} of $M$. Note that by functoriality, this complex can be considered as a complex of discrete modules over the topological algebra $\fZ$.

\begin{rem}

This functor admits a finite-dimensional analogue: see \cite{kostant}. The Drinfeld-Sokolov functor is a kind of analogue of the functor of Whittaker vectors from $p$-adic representation theory. 

\end{rem}

Because $\chi|_{\fn[[t]]}=0$ and by construction of semi-infinite cohomology, there is a map of complexes:
\[
M^{\fn[[t]]}\lto\Psi(M)
\]
\noindent where $M^{\fn[[t]]}$ is regarded as a complex concentrated in degree zero. This map is functorial in $M$.

\subsection{} We have the following essential computation: 

\begin{thm}\label{t:ds}
$\Psi(\bV)$ has cohomology only in degree zero. Moreover, the natural map:
\[
\fz=\bV^{\fg[[t]]}\into\bV^{\fn[[t]]}\lto\Psi(\bV)
\]
\noindent is a quasi-isomorphism of $\fZ$-modules.
\end{thm}

\begin{rem}
In cases such as $\bV$ in which $\Psi(\bV)$ is concentrated in cohomological degree zero, we will not distinguish
notationally between $\Psi(\bV)$ and $H^0(\Psi(\bV))$.
\end{rem}

\begin{proof}

By \cite{bzf} 15.1.9, this cohomology lives only in degree $0$. Moreover, \emph{loc. cit.} shows that the $L_0$-character of 
$\Psi(\bV)$ exists and is equal to the $L_0$-character of $\fz$. Therefore, it suffices to show that the map
$\fz\lto\Psi(\bV)$ is injective. 

To prove this, we appeal to the ``big" Wakimoto module construction. Recall that this construction gives a $\fgh$-module $\bW$ with
an embedding $\bV\into\bW$ (this is the vacuum module associated to a chiral algebra constructed e.g. in \cite{fg2} Section 10, where it is notated $\fD^{ch}(\overset{\circ}{G/B})_{crit}$). 
Since $\bV$ is a submodule of $\bW$, the map $\fz=\bV^{\fg[[t]]}\lto\bW^{\fg[[t]]}$ is an injection.

By (the proof of) \cite{fg2} 12.4.1, the composition $\bW^{\fg[[t]]}\into \bW^{\fn[[t]]}\lto\Psi(\bW)$ is a quasi-isomorphism. This gives a commuting diagram:
\[\xymatrix{
\fz\ar[d]\ar@{^(->}[r] & \bW^{\fg[[t]]}\ar[d]^{\simeq}\\
\Psi(\bV) \ar[r] & \Psi(\bW).
}\]
\noindent As the map $\fz\lto\Psi(\bW)$ is injective, the same must be true of $\fz\lto\Psi(\bV)$. 

\end{proof}

\subsection{} We now have the following result:

\begin{prop}\label{p:psi-exact}
The functor $\Psi:\fgh\modules_{reg}^{G(O)}\lto\fz\modules$ is exact and faithful.
\end{prop}

\begin{proof}

Let $M$ be a regular $G(O)$-integrable $\fgh$-module. We need to show that $\Psi(M)$ is concentrated in degree $0$. First, note that by
$\fz$-linearity, this is true by Theorem \ref{t:ds} for modules of the form $N\otimes_{\fz}\bV$.

Since $\Psi$ commutes with colimits bounded uniformly from below (by definition of semi-infinite cohomology) and because the $t$-structure on $\fgh\modules$ is compatible with filtered
colimits, Zorn's lemma implies that there exists a submodule $M_0$ of $M$ maximal with respect to inclusion among submodules $M'$ of $M$ with the property that $\Psi(M')$ is concentrated in
cohomological degree $0$. 

But if $M_0$ is not equal to $M$, then
$M/M_0$ has a non-zero submodule of the form $N\otimes_{\fz}\bV$ by Corollary \ref{c:injection}, and $\Psi(N\otimes_{\fz}\bV)=N$
contradicts the maximality of $M_0$.

Now let us deduce that $\Psi$ is faithful. We have $M^{G(O)}\otimes_{\fz}\bV\into M$ by Corollary \ref{c:injection}. Applying $\Psi$, we see that $M^{G(O)}\lto \Psi(M)$ is also injective. 
The statement of the proposition then follows because $M^{G(O)}\neq 0$ by Proposition \ref{p:generator}.

\end{proof}

\subsection{}\label{ss:finiteness} Let $V\in\Rep(\ld{G})$ be a finite-dimensional representation. Consider the corresponding 
critically twisted $D$-module $\fF_V$ on $\Gr_G$. Note that $\fF_V$ has finite-dimensional support, so $R\Ga(\Gr_G,\fF_V)$ lives in
bounded cohomological degrees. 

\begin{fthm}\label{t:finiteness}
The object $\Psi(R\Ga(\Gr_G,\fF_V))\in D^b(\fz\modules)$ is perfect, i.e., quasi-isomorphic to a bounded complex of finite rank projective $\fz$-modules.
\end{fthm}

This will be proved in Section \ref{s:finiteness}.

\subsection{} Here is an immediate corollary of the Finiteness Theorem:

\begin{cor}
Let $A$ be a
$\fz$-algebra which is coherent as a ring (e.g., this is satisfied for $A=\fz$ and when $A$ is Noetherian).
Then for all $i$, the $A$-module: 
\[
H^i(\Psi(R\Ga(\Gr_G,\fF_V))\underset{\fz}{\overset{L}{\otimes}}A)=\Psi(H^i(R\Ga(\Gr_G,\fF_V)\underset{\fz}{\overset{L}{\otimes}}A))
\]
\noindent is finitely presented.
\end{cor}

\section{Proof of the main theorem}\label{s:vanishing}

\subsection{} The goal for this section is to explain how to deduce the following result from the Finiteness Theorem:

\begin{thm}\label{t:main}
For $V\in\Rep(\ld{G})$, $H^i(\Gr_G,\fF_V)=0$ for $i>0$ and $\Ga(\Gr_G,\fF_V)^{G(O)}$ is a finitely generated projective $\fz$-module.
\end{thm}

The proof will occupy the entirety of this section. 

\begin{rem} 

This theorem was proved originally in \cite{bd} 5.4.8. However, their proof relied on the Feigin-Frenkel isomorphism. More precisely, in the proof of their Lemma 6.2.2, they appeal to 
the surjectivity of the anchor map for the algebroid coming from the chiral Poisson structure on (the VOA form of) $\fz$. They deduce the surjectivity of this anchor map 
from the compatibility of the Feigin-Frenkel isomorphism with chiral Poisson structures, since the corresponding algebroid on the opers side is much more concrete. It appears infeasible to give a 
direct proof of this surjectivity of this anchor map.

\end{rem}

\subsection{}\label{ss:main-start} We will need the following result from commutative algebra:

\begin{prop}\label{p:polynomials}
Let $A$ be a polynomial algebra on countably many generators over a Noetherian ring $R$. Then a prime ideal $\fp$ of
$A$ has finite height if and only if it is finitely generated.
\end{prop}

\begin{proof}

That finite height of a prime ideal in $A$ implies finite generation is proved in \cite{gilmer-heinzer}.

For the converse, we claim that the height of $\fp$ is bounded by the number of generators of the ideal. Suppose $\fp$ is generated by $n$ elements. Expanding $R$ as necessary,
we may assume that all the generators lie in $R$. Let $\fp=\fp_0\supset \fp_1\supset \fp_2\supset \ldots \supset \fp_{n+1}$ be a chain of prime ideals in $A=R[X_1,X_2,\ldots]$. 
Then we need to show that $\fp_i=\fp_{i+1}$ for some $i$.

For each $r\in\bZ^{\geq 0}$, let $\fp_i^{(r)}:=\fp_i\cap R[X_1,\ldots,X_r]$. Since $R[X_1,\ldots,X_r]$ is a Noetherian ring and since $\fp_0^{(r)}$ is generated by $n$ elements
for all $r$, Krull's theorem tells us that for each $r$ there exists $i$ such that $\fp_i^{(r)}=\fp_{i+1}^{(r)}$. We deduce that there exists $i$ such that
$\fp_i^{(r)}=\fp_{i+1}^{(r)}$ for \emph{all} $r$. This implies that $\fp_i=\fp_{i+1}$ since $\underset{r}{\cup} R[X_1,\ldots,X_r]=A$.

\end{proof}

\subsection{}

Proposition \ref{p:polynomials} in particular implies the following:

\begin{cor}\label{c:noeth}
Let $A$ be a polyonomial algebra on countably many generators over $\bC$. Then the localization $A_{\fp}$ of $A$ at any finite height prime ideal $\fp$ is 
a regular local ring (in particular Noetherian).
\end{cor}

\begin{proof}

Recall Cohen's theorem that a commutative ring is Noetherian if and only if every \emph{prime} ideal is finitely generated. But this follows from Proposition \ref{p:polynomials}.

Now regularity follows from the corresponding statement in finitely generated polynomial algebras. For example, one can proceed as follows. Let $\fp$ be a prime ideal of finite height, which by 
Proposition \ref{p:polynomials} is finitely generated. As $A/\fp$ is finitely presented, Lemma \ref{l:res} implies that it admits a finite resolution by finite rank free $\fz$-modules. Localizing at
$\fp$, we deduce the same result for the residue field $A_{\fp}/\fp\cdot A_{\fp}$, which immediately implies regularity.

\end{proof}

\subsection{} We will prove the following statement by induction on $n$:

\begin{description}

\item[$(*)_n$] For every $\fp\in\Spec(\fz)$ of height $n<\infty$ and for every finite-dimensional representation $V$ of $\Rep(\ld{G})$, $H^i(\Gr_G,\fF_V)_{\fp}=0$ for 
$i>0$ and $\Psi(H^0(\Gr_G,\fF_V)_{\fp})$ is a finite rank projective (equivalently, free) $\fz_{\fp}$-module.

\end{description}

The proof occupies Sections \ref{ss:induction-start}-\ref{ss:induction-finish}.

\subsection{}\label{ss:induction-start} We will appeal to the convolution format of \cite{bd} Section 7, so we rewrite
$R\Ga(\Gr_G,\fF_V)$ as $V*\bV$ (considered as an object of the appropriate derived category of regular $G(O)$-integrable $\fgh$-modules). We will also use the convolution format
``with extra parameters" as in \emph{loc. cit.} 7.9.8 (see also \emph{loc. cit.} 5.4.9).

Let $\bV_{\fp}=\bV\otimes_{\fz}\fz_{\fp}$ be the $\fz_{\fp}$-linear version of $\bV$. 
Then, understanding $V*\bV_{\fp}$ as in \emph{loc. cit.} (it is considered as an object of the derived category of $\fz_{\fp}$-linear $G(O)$-integrable $\fgh$-modules), we have 
$\Psi(H^k(V*\bV_{\fp}))\simeq \Psi(H^k(V*\bV))_{\fp}$ by $\fz$-linearity of the functors involved.

\subsection{} We begin in Section \ref{ss:base-case} with the base case of the claim $(*)_n$ where $n=0$ (i.e., $\fp=0$). 

The argument in this case was taught to us by Gaitsgory in discussions surrounding his 
2009-10 seminar (and was the starting point for the present work). A version of this argument
was found by Beilinson and Drinfeld, but was not written in the unfinished work \cite{bd} (however, a variant does appear as Lemma 5.5.2). The proof of \cite{fusion} Theorem 4.14 
is another instance of this argument.

\subsection{}\label{ss:base-case} Since $\fz_{0}=\fz_{\fp}$ is a field, $H^i(V*\bV_{0})^{G(O)}$ is free over $\fz_{0}$ for all $i$. By Corollary \ref{c:vac-ext-param},
we have the isomorphism $H^i(V*\bV_{0})^{G(O)}\otimes_{\fz_0}\bV_0\isom H^i(V*\bV_0)$.

Let $k\geq 0$ be the maximal integer for which $H^k(V*\bV_0)\neq 0$. Then the map 
$V*\bV_{0}\lto H^k(V*\bV_{0})[-k]$ is non-zero. Therefore, with $V^{\vee}$ the dual representation to $V$, the adjoint map:
\[
\bV_0\lto V^{\vee}*H^k(V*\bV_0)[-k]
\]
\noindent is non-zero. 

Note that $V^{\vee}*H^k(V*\bV_0)$ lives in non-negative cohomological degrees because $H^k(V*\bV_0)$ is isomorphic to a direct sum of copies of $\bV_0$,
and $V^{\vee}*\bV_0=(V^{\vee}*\bV)_0=R\Ga(Gr_G,\fF_{V^{\vee}})_0$ lives in non-negative cohomological degrees. 

Therefore, $V^{\vee}*H^k(V*\bV_0)[-k]$ sits in cohomological degrees $\geq k$. But because $\bV_0$ lives in cohomological degree zero, no such map can exist unless $k=0$.
Thus, higher cohomology must vanish, and we have already seen that $H^0(V*\bV_0)^{G(O)}$ is free over $\fz_0$ and therefore coincides with $\Psi(H^0(V*\bV_0))$.

\subsection{} Now fix $\fp$ a prime ideal of $\fz$ with $1\leq n=\on{height}(\fp)<\infty$. We
assume that the inductive statement is proved for all prime ideals properly contained in $\fp$ and will deduce below that it is satisfied for $\fp$.

We let $k_{\fp}$ denote the residue field of $\fz_{\fp}$.

\subsection{} Let $\bV_{k_{\fp}}=\bV\otimes_{\fz}k_{\fp}$ and let $V*\bV_{k_{\fp}}$ be as in Section \ref{ss:induction-start}, i.e., it is an object of the appropriate $G(O)$-equivariant derived category of
$\fgh\modules_{k_{\fp}}$ given by the convolution format ``with extra parameters." There is a natural quasi-isomorphism $(V*\bV_{\fp})\overset{L}{\otimes}_{\fz_{\fp}}k_{\fp}\lto V*\bV_{k_{\fp}}$ coming
from the convolution format (here we use the flatness of $\bV$ as a $\fz$-module). We will regard these objects as objects of the $G(O)$-equivariant derived category of 
$\fgh\modules_{\fz_{\fp}}$ when convenient.

\subsection{} Let $k$ be the highest degree in which $H^k(V*\bV_{\fp})\neq 0$. Suppose for the sake of contradiction that $k\geq{1}$.

\subsection{} Recall that a finitely generated module over a Noetherian ring has non-trivial support. Therefore, by the Finiteness Theorem and because $\fz_{\fp}$ is Noetherian, the 
support of $\Psi(H^k(V*\bV_{\fp}))$ must be non-zero. By the inductive hypothesis and because $k\geq{1}$, it therefore must be exactly $\fp$.
But this means that a sufficiently large power of $\fp$ annihlates $\Psi(H^k(V*\bV_{\fp}))$. Equivalently, $\Psi(H^k(V*\bV_{\fp}))$ is a finite length $\fz_{\fp}$-module. 

Because $\Psi$ is faithful and exact, both of the properties ``lift" to similar properties of $H^k(V*\bV_{\fp})$, i.e., this module is annihlated by large powers of $\fp$ and is finite length in $\fgh\modules_{reg}^{G(O)}$.
In particular, by Corollary \ref{c:injection}, $H^k(V*\bV_{\fp})$ admits a filtration by modules of the form $N_i\otimes_{\fz_{\fp}}\bV_{\fp}$ where $N_i$ is a $\fp$-torsion $\fz_{\fp}$-module. Considering the
top term in this filtration, it follows that $H^k(V*\bV_{\fp})$ admits a surjection to $\bV_{k_{\fp}}$.

\subsection{}\label{ss:bound} Next, we claim that $k\leq{n}$, where we recall that $n=\on{height}(\fp)$.

Indeed, by the above $H^k(V*\bV_{\fp})$ admits a surjection to $\bV_{k_{\fp}}$. Therefore, we have a non-zero map:
\[
V*\bV_{\fp}\lto \bV_{k_{\fp}}[-k]
\]
which gives rise by the adjunction to a non-zero map:
\[
\bV_{\fp}\lto V^{\vee}*\bV_{k_{\fp}}[-k].
\]
\noindent Therefore, it suffices to show that $V^{\vee}*\bV_{k_{\fp}}$ is concentrated in cohomological degrees $\geq{-n}$. 

But this follows immediately by considering the Koszul resolution of $k_{\fp}$ over $\fz_{\fp}$ because 
$V^{\vee}*\bV_{\fz_{\fp}}=R\Ga(\Gr_G,\fF_{V^{\vee}})_{\fp}$ is concentrated in non-negative cohomological degrees.

\subsection{}\label{ss:tor-sseq} We have a spectral sequence:
\[
E_1^{pq}=\Tor_{-q}^{\fz_{\fp}}(H^p(V*\bV_{\fp}),\fz_{\fp})\Rightarrow H^{p+q}((V*\bV_{\fp})\underset{\fz_{\fp}}{\overset{L}{\otimes}}k_{\fp}).
\]
\noindent This spectral sequence converges because $\fz_{\fp}$ is a regular local ring and therefore $\Tor$-finite. 
This spectral sequence implies that $H^k(V*\bV_{\fp})\otimes_{\fz_{\fp}}k_{\fp}\isom H^k((V*\bV_{\fp})\otimes_{\fz_{\fp}}^Lk_{\fp})$.

Note that $\Psi(H^k(V*\bV_{\fp})\otimes_{\fz_{\fp}}k_{\fp})=\Psi(H^k(V*\bV_{\fp}))\otimes_{\fz_{\fp}}k_{\fp}$ is a finite-dimensional vector space by the Finiteness Theorem. Moreover,
we deduce by Nakayama's lemma that it is non-zero because $\Psi(H^k(V*\bV_{\fp}))$ is.

\subsection{} Now we change notation somewhat: let $k$ denote the largest integer for which $H^k(V*\bV_{\fp})\neq{0}$ where we no longer
fix the finite dimensional representation $V$, i.e., it is the largest integer for which $H^{k'}(W*\bV_{\fp})=0$ for all $k'>k$ and for all finite dimensional $W\in\Rep(\ld{G})$. 
Note that $k\leq{n}<\infty$ by Section \ref{ss:bound}. Again, we suppose for the sake of contradiction that $k\geq{1}$.

By Section \ref{ss:bound}, we have $H^{k'}(W*\bV_{k_{\fp}})=0$ for $k'\not\in[-n,k]$.

\subsection{} We begin by showing $H^k(W*H^k(V*\bV_{k_{\fp}}))=0$ for every finite dimensional $V,W\in\Rep(\ld{G})$. 

We have a Grothendieck spectral sequence:
\[
E_2^{pq}=H^p(W*H^q(V*\bV_{k_{\fp}}))\Rightarrow H^{p+q}((W\otimes{V})*\bV_{k_{\fp}}).
\]
\noindent All of the cohomology groups above vanish outside of degrees $[-n,k]$. Therefore $E_2^{kk}=E_{\infty}^{kk}$, and this is zero because $H^{2k}((W\otimes V)*\bV_{k_{\fp}})=0$ (because
we assume $k>0$ and therefore $2k>k$). 

\subsection{} Suppose $V$ is chosen so that $H^k(V*\bV_{\fp})\neq{0}$. As before, the module $H^k(V*\bV_{\fp})$ is equipped with a surjection to $\bV_{k_{\fp}}$. Let $K$ denote the kernel of this map. 

Note that $\Psi(K)$ is also $\fp$-torsion and therefore of finite length with a filtration by $\bV_{k_{\fp}}$. Therefore $H^i(V*K)=0$ for $i>k$ by d\'evissage.

\subsection{} Applying convolution by $V$ to the short exact sequence:
\[
0\lto K\lto H^k(V*\bV_{k_{\fp}})\lto \bV_{k_{\fp}}\lto 0
\]
\noindent and taking $k$th cohomology, we get the exact sequence:
\[
0=H^k(V*H^k(V*\bV_{\fp}))\lto H^k(V*\bV_{k_{\fp}})\lto H^{k+1}(V*K)=0.
\]
\noindent Therefore, we deduce that $H^k(V*\bV_{k_{\fp}})=0$. But we've seen in Section \ref{ss:tor-sseq} that $H^k(V*\bV_{k_{\fp}})$ is non-zero, so we have a contradiction.

Thus, $H^i(V*\bV_{\fp})=0$ for all $i>0$ and all $V\in\Rep(\ld{G})$.

\subsection{} It remains to show that $\Psi(V*\bV_{\fp})=H^0(\Psi(V*\bV_{\fp}))$ is a projective $\fz_{\fp}$-module. By the Finiteness Theorem, this $\fz_{\fp}$-module is finitely generated. Therefore, we need
only to see that it is flat. 

It suffices to show that the complex:
\[
\Psi(V*\bV_{\fp})\underset{\fz_{\fp}}{\overset{L}{\otimes}}k_{\fp}
\]
\noindent is concentrated in cohomological degree $0$, i.e., that no $\Tor$ is formed when tensoring with the residue field $k_{\fp}$. By linearity of $\Psi$, we have a quasi-isomorphism: 
\[
\Psi(V*\bV_{\fp})\underset{\fz_{\fp}}{\overset{L}{\otimes}}k_{\fp}\isom \Psi((V*\bV_{\fp})\underset{\fz_{\fp}}{\overset{L}{\otimes}}{k_{\fp}}).
\]
\noindent By the convolution format, the right hand side is $\Psi(V*\bV_{k_{\fp}})$. Therefore it suffices to see that $\Psi(V*\bV_{k_{\fp}})$ is concentrated in cohomological degree $0$,
i.e., that this is true of $V*\bV_{k_{\fp}}$.

\subsection{}\label{ss:induction-finish} We have shown that 
$V*\bV_{k_{\fp}}$ is concentrated in non-positive cohomological degrees for all $V$, and that this
object is cohomologically bounded below (by $-n$). Let $k\in\bZ^{\leq 0}$ be the smallest integer for which $H^k(V*\bV_{k_{\fp}})\neq{0}$. Note that there is a non-zero map $\bV_{k_{\fp}}\lto H^k(V*\bV_{k_{\fp}})$. This gives us a non-zero map:
\[
\bV_{k_{\fp}}[-k]\lto V*\bV_{k_{\fp}}
\]
\noindent which by adjunction gives rise to a non-zero map:
\[
V^{\vee}*\bV_{k_{\fp}}[-k]\lto \bV_{k_{\fp}}.
\]
\noindent Here the left hand side is concentrated in degrees $\leq{k}$ and the right hand side is concentrated in degree $0$. Since $k\leq 0$ we must have $k=0$ as desired.

This completes the inductive step, and therefore we deduce that $(*)_n$ is true for all $n$.

\subsection{} To proceed, we will need the following lemma:

\begin{lem}\label{l:height}
For any non-zero finitely presented $\fz$-module $N$, there exists a finite height prime ideal $\fp$ of $\fz$ such that the localization $N_{\fp}$ is non-zero.
\end{lem}

\begin{proof}

There exists $\fz'\subset\fz$ a sub-algebra of finite type over $\bC$ and a finitely generated $\fz'$-module $N_0$ such that
$N\simeq N_0\otimes_{\fz'}\fz$ (to produce such data, choose a finite presentation of $N$, let $\fz'$ be the subalgebra generated by the coefficients appearing
in the matrix for the presentation, and then define $N_0$ by the ``same" presentation over $\fz'$). Because $\fz$ is a polynomial algebra, expanding $\fz'$ as necessary we can
assume that $\fz$ is obtained from $\fz'$ by adjoining (countably many) generators freely.
Let $\fp_0$ be a prime ideal of $\fz'$ in the support of $N_0$ (e.g., any associated prime will do). Of course, $\fp_0$ is finitely generated. Let $\fp:=\fp_0\otimes_{\fz'}\fz$.
Because $\fz$ is a polynomial algebra over $\fz'$, $\fp$ is prime in $\fz$. So we are done by Proposition \ref{p:polynomials}. 

\end{proof}

\subsection{} Now we are set to show that $H^i(V*\bV)=0$ for $i>0$.

We immediately reduce to proving the vanishing of cohomology in the case when $V$ is finite-dimensional. Suppose $k$ is the largest degree for which 
$H^{k}(V*\bV)\neq 0$. By Proposition \ref{p:psi-exact}, $\Psi(H^{k}(V*\bV))\neq 0$.

We have already seen that the localization $H^i(V*\bV)_{\fp}$ at a finite height prime $\fp$ is zero for all $i>0$.
But by the Finiteness Theorem, this gives a contradiction to Lemma \ref{l:height} applied to $\Psi(H^{k}(V*\bV))$ unless $k=0$.

\subsection{} Finally, to complete the proof of Theorem \ref{t:main}, it remains to show that $H^0(V*\bV)^{G(O)}$ is projective over $\fz$.
Towards this end, we first show instead that\footnote{The notation may seem a bit strange: we write $H^0(V*\bV)^{G(O)}$ but $\Psi(V*\bV)$ without an $H^0$. This is to emphasize that we are taking 
$G(O)$-invariants in the underived sense, while there is no such ambiguity with $\Psi$.} $\Psi(V*\bV)$ is a finite rank projective $\fz$-module. 

The Finiteness Theorem implies that $\Psi(V*\bV)$ is finitely presented over $\fz$, so it suffices to show that this module is flat. For this, it is enough to show that
for any finitely presented $\fz$-module $N$ we have
$\Tor_i(\Psi(V*\bV),N)=0$ for $i>0$. 

Note that the $\fz$-module $\Tor_i(\Psi(V*\bV),N)$ is finitely presented because $\fz$, being an infinite polynomial algebra, is coherent (indeed, $\Psi(V*\bV)$ and $N$ are finitely presented and therefore 
coherence implies that they admit resolutions by finite free $\fz$-modules). But for
$i>0$ the localizations of these groups at finite height prime ideals is $0$, so by Lemma \ref{l:height} they must vanish as desired. Therefore $\Psi(V*\bV)$ is flat and finitely presented and therefore projective.

\subsection{} Let $\fp$ be a height one prime ideal of $\fz$. In this case we will show that $H^0(V*\bV_{\fp})^{G(O)}$ is a flat $\fz_{\fp}$-module.

By Proposition \ref{p:polynomials} we know that $\fz_{\fp}$ is a DVR. Therefore we need only to see that $H^0(V*\bV_{\fp})^{G(O)}$ is torsion-free. 

If it were not, we would have an injection $\bV_{k_{\fp}}\into V*\bV_{\fp}$. By adjunction, this gives rise to a non-zero map $V^{\vee}*\bV_{k_{\fp}}\lto \bV_{\fp}$. However, the left hand side is annihlated by $\fp$ while the right hand side
has no $\fp$-torsion, so we have a contradiction.

\subsection{}\label{ss:psi-diagram} We have the following diagram:
\[\xymatrix{
H^0(V*\bV)^{G(O)}\otimes_{\fz}\bV\ar[rr]\ar[d] && V*\bV \\
\Psi(V*\bV)\otimes_{\fz}\bV.
}\]
Moreover, after at any height one prime ideal of $\fz$ the maps in this diagram become isomorphisms.
Indeed, the localizations of $H^0(V*\bV)^{G(O)}$ at such primes is free so we get the desired result by applying Corollary \ref{c:vac-ext-param}.

\subsection{} We have the following (simple) lemma from commutative algebra (whose statement we were unable to find in the literature):

\begin{lem}\label{l:map-extn}
Let $A$ be a (possibly non-Noetherian) UFD and let $M$ and $N$ be two $A$-modules with $M$ flat over $A$. Suppose that for
each height one prime ideal $\fp$ of $A$ we are given morphisms $f_{\fp}:N_{\fp}\lto M_{\fp}$ which are compatible in the sense that the induced morphisms
over the generic point of $A$ give rise to the same map $f_0:N_0\lto M_0$ (here e.g. $N_0$ denotes the localization at the generic point). Then
the map $f_0$ extends uniquely to an $A$-linear map $f:N\lto M$ such that the following diagram commutes:
\[\xymatrix{
N\ar[rr]^{f}\ar[d] && M\ar[d] \\
N_0\ar[rr]^{f_0} && M_0.
}\]
\end{lem}

\begin{rem}
The argument below shows that it is enough to have the maps $f_{\fp}$ only when $\fp$ is a \emph{principal} prime ideal of height 1, though for $A=\fz$ all height 1 prime ideals are principal by Proposition \ref{p:polynomials}. 
\end{rem}

\begin{proof}

For every prime ideal $\fp$ of $A$, the map $M\lto M_{\fp}$ is an injection because $M$ is flat and because $A$ is a domain. Applying this in the case $\fp=0$, we immediately deduce
the uniqueness of any such extension. To see existence, it suffices to show:
\[
M=\underset{\on{height}(\fp)=1}{\bigcap}M_{\fp}.
\]
\noindent Clearly the left hand side is contained in the right hand side. 

Let $x\in\underset{\on{height}(\fp)=1}{\bigcap}M_{\fp}$. For some $0\neq f\in A$, we have $fx\in M$. Let $p$ be a prime element of $A$ and choose $n\in\bZ$ so that $p^n\mid f$ and 
$g:=\frac{f}{p^n}$ is not divisible by $p$.
We will show that $gx\in M$, so that by iterating the argument over primes dividing $f$ we deduce that $x\in M$ as desired.

Because $x\in M_{(p)}\subset M_0$, for some $h\in A$ prime to $p$ we have $hx\in M$. Then we have:
\[
h\cdot (fx)-p^n\cdot(ghx)=0
\]
\noindent Note that here $fx$ and $ghx$ are elements of $M$. Because $h$ is prime to $p$, the pair $(h,p^n)$ forms a regular sequence, and therefore flatness of $M$ and 
consideration of the Koszul complex for $(h,p^n)$ imply that there exists
some $y\in M$ such that $p^ny=fx$ and $hy=fx$. But clearly $y=gx$ because $M_0$ is torsion-free and $gx$ satisfies the above equations, so we deduce that $gx\in M$.

\end{proof}

\subsection{} We claim that there is a unique $\fgh$-module map $V*\bV\lto \Psi(V*\bV)\otimes_{\fz}\bV$ making the diagram from Section \ref{ss:psi-diagram} commute. 

Indeed, $\Psi(V*\bV)\otimes_{\fz}\bV$ is flat over $\fz$ (since we have seen this for both factors). Therefore, applying Lemma \ref{l:map-extn} and the observation from
Section \ref{ss:psi-diagram}, we get a $\fz$-linear map as desired, and this is immediately seen to be a $\fgh$-module map by the uniqueness statement of the lemma.

\subsection{}\label{ss:main-finish} To complete the proof of Theorem \ref{t:main}, it suffices to show that the map $V*\bV\lto\Psi(V*\bV)\otimes_{\fz}\bV$ is an isomorphism. It suffices
to see this after applying the functor $\Psi$. More precisely, we claim that the induced map $\Psi(V*\bV)\lto\Psi(V*\bV)$ is the identity. Indeed, it suffices to check this at the generic point,
where it follows by construction.

\section{Birth of opers}\label{s:opers}

\subsection{} In this section, we recall the ``birth of opers" construction from \cite{bd} (see \emph{loc. cit.} 5.2.18, 5.5). This construction uses Theorem \ref{t:main} to construct an $\Aut^+$-equivariant map:
\[
\vph:\Spec(\fz)\lto \Op_{\ld{G}}(\cD).
\]

The construction is essentially given in \cite{bd}, but some details are not included because the text is unfinished. The construction appears not to be given in a complete form in the literature, so we
include a short exposition of their construction.

\subsection{} Suppose $S$ is a scheme equipped with a map $S\lto \Op_{\ld{G}}(\cD)$. This data defines in particular a $\ld{B}$-bundle on $S$ by Remark \ref{r:bg}.

\subsection{} We have the following observation, proved as \cite{bd} 3.5.8:

\begin{prop}\label{p:oper}
Let $A$ be a commutative algebra with an action of $\Aut^+$ satisfying the following conditions:
\begin{enumerate}

\item The eigenvalues of $L_0$ acting on $A$ are non-negative.

\item The eigenvalue $0$ of $L_0$ occurs with multiplicity one.

\item The eigenvalue $1$ of $L_0$ occurs with multiplicity zero.

\end{enumerate}

Suppose $\Spec(A)$ is equipped with an $\Aut^+$-equivariant $\ld{G}$-bundle $\cQ_{\ld{G}}$
and with a $\Aut$-equivariant reduction $\cQ_{\ld{B}}$ to $\ld{B}$. Then this bundle comes as the restriction from an $\Aut^+$-equivariant map $\Spec(A)\lto\Op_{\ld{G}}(\cD)$ if and 
only if the $\Aut$-equivariant $\ld{H}$-bundle on $\Spec(A)$ induced via $\ld{B}\lto\ld{H}$ is obtained by pull-back from 
$\Spec(\bC)$ of the $\Aut$-equivariant $\ld{H}$-bundle $\rho(\omega_{\cD}/t\omega_{\cD})$.

\end{prop}

\begin{rem}
That the $\ld{B}$-reduction is merely $\Aut$-equivariant encodes the fact that the $\ld{B}$-reduction of an oper is not preserved by the connection.
\end{rem}

One readily verifies that the algebra $\fz$ satisfies the conditions of the proposition (it suffices to verify these conditions at the associated graded level). Therefore, to 
define $\vph$ it suffices to define an $\Aut^+$-equivariant $\ld{G}$-bundle on $\Spec(\fz)$
with an $\Aut$-equivariant reduction to $\ld{B}$ such that the induced $\ld{H}$-bundle is as in the proposition. 

\subsection{} We begin with the construction of an $\Aut^+$-equivariant $\ld{G}$-bundle on $\Spec(\fz)$. By Tannakian formalism, to give such a bundle amounts to giving an 
$\Aut^+$-equivariant symmetric monoidal functor:
\[
\cP:\on{Rep}(\ld{G})\lto \fz\modules
\]
\noindent where $\on{Rep}(\ld{G})$ has the trivial $\Aut^+$-action.

\subsection{} We define: 
\[
\cP(V):=H^0(V*\bV)^{G(O)}=\Ga(\Gr_G,\fF_V)^{G(O)}.
\]
\noindent By Theorem \ref{t:main}, $\cP(V)$ is a projective $\fz$-module. Let us upgrade $\cP$ to a monoidal functor. We need to give an isomorphism:
\[
\alpha_{VW}:\cP(V\otimes W)\isom \cP(V)\underset{\fz}{\otimes}\cP(W).
\]

Corollary \ref{c:invariants}, the natural map:
\[
\cP(V)=H^0(V*\bV)^{G(O)}\underset{\fz}{\otimes}\bV\lto H^0(V*\bV)=V*\bV
\]
\noindent is an isomorphism. By the convolution formalism, we have an isomorphism:
\[
(V\otimes W)*\bV\isom V*(W*\bV).
\]
\noindent By the above, the right hand side is $V*(\cP(W)\otimes_{\fz}\bV)$, and this is canonically isomorphic to $\cP(W)\otimes_{\fz}(V*\bV)$ by linearity of convolution. 
Therefore, applying $G(O)$-invariants to the isomorphism above gives:
\[
\cP(V\otimes W)=\bigg((V\otimes W)*\bV\bigg)^{G(O)}\isom \bigg(V*(W*\bV)\bigg)^{G(O)}\isom
\]
\[
\isom\cP(W)\underset{\fz}{\otimes}(V*\bV)^{G(O)}=\cP(W)\underset{\fz}{\otimes}\cP(V)=\cP(V)\underset{\fz}{\otimes}\cP(W).
\]
\noindent Then we define $\alpha_{VW}$ as this composition.

One immediately verifies that $\alpha_{VW}$ satisfies the associativity condition and therefore defines a monoidal functor. That it is symmetric monoidal
is verified as in \cite{bd} 5.5.5-6.

\subsection{} Next, we describe an $\Aut$-equivariant reduction of $\cP$ to a $\ld{B}$-bundle. 

We fix a uniformizer $t$ on $\cD$. Let $L_i=-t^{i+1}\cdot\frac{\partial}{\partial t}\in\Lie(\Aut^+)$.

\begin{lem}\label{l:key}
With $V^{\check{\lambda}}$ the irreducible representation of $\ld{G}$
of highest weight $\check{\lambda}$, the lowest eigenvalue of $L_0$ acting on $\Ga(\Gr_G,\fF_{V^{\check{\lambda}}})$ is $-\check{\la}(\rho)$. This eigenspace is one-dimensional over $\bC$.
\end{lem}

This is proved as \cite{bd} 8.1.5.

\subsection{}\label{ss:canonical} Let $\ell^{\check{\la}}\subset \Ga(\Gr_G,\fF_{V^{\check{\lambda}}})$ be the eigenspace described by the lemma. First, we claim that
this line is independent of the choice of uniformizer $t$. It suffices to show that $\ell^{\check{\la}}$ is fixed by the kernel $\Ker(\Aut\lto\bG_m)$ of
the standard character (given by taking the derivative of an automorphism at $0$) because then $\Aut$ acts on such invariants through $\bG_m$ and the operator $L_0$ becomes canonical. 

Because $[L_0,L_i]=-iL_i$, $L_0$ acts on $L_i\cdot \ell^{\check{\la}}$ with eigenvalue $-\check{\la}(\rho)-i$. Thus, for $i>0$ Lemma \ref{l:key} implies that $L_i\cdot\ell^{\check{\la}}=0$.

\subsection{} Next, we prove that $\ell^{\check{\la}}$ lies in the $G(O)$-invariants $\Ga(\Gr_G,\fF_{V^{\check{\la}}})^{G(O)}$ of $\Ga(\Gr_G,\fF_{V^{\check{\la}}})$.

Note that $L_0$ acts on $\fg[[t]]$ with non-positive eigenvalues, acting on $\Ker(\fg[[t]]\to \fg)$ with
strictly negative eigenvalues. Therefore, by the argument above, this kernel acts by $0$ on $\ell^{\check{\la}}$. 

Because $L_0$ acts with the eigenvalue $0$
on $\fg\subset\fg[[t]]$, $\fg$ preserves the eigenspaces of $L_0$. In particular, $\fg$ acts on $\ell^{\check{\la}}$ since this eigenspace is 1-dimensional.
However, $\fg$ has no non-trivial characters because $\fg$ is semisimple. Therefore, $\fg$ acts on $\ell^{\check{\la}}$ trivially. 

Since the embedding of $\fg\into\fg[[t]]$ induces an isomorphism with the quotient $\fg$ of $\fg[[t]]$, this implies that $\fg[[t]]$ acts on $\ell^{\check{\la}}$ trivially.

\subsection{} We recall the Plucker construction, which we will use to define the reduction of $\cP$ to $\ld{B}$. 

The Plucker construction says that to define a reduction of our $\ld{G}$-bundle $\cP$ to $\ld{B}$, it is enough to give the following data. For each 
dominant weight $\check{\lambda}$ of $\ld{G}$,
we are required to specify a line bundle $\sL^{\check{\la}}\subset \cP(V^{\check{\la}})$ with vector bundle quotient. 
Furthermore, we need to specify isomorphisms $\sL^{\check{\la}+\check{\la}'}\isom \sL^{\check{\la}}\otimes_{\fz}\sL^{\check{\la}'}$ which ``satisfy the Plucker relations."
That is, we demand that the following diagram commute:
\[\xymatrix{
\sL^{\check{\la}+\check{\la}'}\ar[d]\ar[rr] && \sL^{\check{\la}}\underset{\fz}{\otimes}\sL^{\check{\la}'}\ar[d] \\
\cP(V^{\check{\la}+{\la'}})\ar[r] & \cP(V^{\check{\la}}\otimes V^{\check{\la}'})\ar[r]^{\alpha_{VW}} & \cP(V^{\check{\la}})\underset{\fz}{\otimes}\cP(V^{\check{\la}'}) 
}\]
\noindent where here we are taking the natural map $V^{\check{\la}+{\la'}}\into V^{\check{\la}}\otimes V^{\check{\la}'}$.

\subsection{} We have constructed ($\bC$-)lines $\ell^{\check{\la}}$ inside of $\cP(V^{\check{\la}})$. Since $\fz$ is a domain, the
$\fz$-module generated by $\ell^{\check{\la}}$ is free by Theorem \ref{t:main}.
Furthermore, comparing the $L_0$-eigenvalues and using Lemma \ref{l:key}, we see that these lines
satisfy the Plucker relations automatically. 

These line bundles are evidently $\Aut$-equivariant. Therefore, to give our $\Aut$-equivariant reduction to $\ld{B}$, it suffices to show
that the quotients of these vector bundles by the line bundles are projective $\fz$-modules.

\subsection{} Obviously the quotient of $\cP(V^{\check{\la}})$ by $\sL^{\check{\la}}$ is finitely presented, so to verify projectivity
it suffices to show that this quotient is flat. 

Our choice of uniformizer defines a splitting $\bG_m\into\Aut$ and therefore defines a grading of $\fz$. Since $\fz$ is in non-positive\footnote{Note that the $\bG_m$ grading is 
opposite to the $L_0$-eigenvalue.} degrees with $\bC$ as the degree $0$ part,
it suffices to show that $\Tor_1(\bC,\cP(V^{\check{\la}})/\sL^{\check{\la}})=0$ where $\bC$ is realized as the degree $0$ quotient $\fz\onto\bC$.

To show this $\Tor$ vanishes, we need to show that $\sL^{\check{\la}}\into \cP(V^{\check{\la}})$ remains an injection after tensoring with $\bC$. The induced map
$\ell^{\check{\la}}\lto\cP(V^{\check{\la}})/\fz_{>0}\cdot\cP(V^{\check{\la}})$ is non-zero because
$\ell^{\check{\la}}$ is the highest graded line in $\cP(V^{\check{\la}})$, proving the desired flatness. 

\subsection{} To construct the map $\vph$, it remains to show that the $\Aut$-equivariant $\ld{H}$-bundle on $\Spec(\fz)$ induced by $\cP_{\ld{B}}$ is given as in
Proposition \ref{p:oper}. For a weight $\check{\la}$ of $\ld{G}$, the line 
bundle induced from our $\ld{H}$-bundle is the trivial line bundle on $\Spec(\fz)$ with $\Aut$-action the $\check{\la}(\rho)$-th power of the standard character because of Lemma \ref{l:key},
as desired. 

By Proposition \ref{p:oper}, this defines the desired map:
\[
\vph:\Spec(\fz)\lto\Op_{\ld{G}}(\cD).
\]

\subsection{} To complete the proof of Theorem \ref{t:ff}, it remains to show that $\vph$ is an isomorphism. Indeed, $\vph$ otherwise satisfies the compatibilities of Theorem \ref{t:ff} by construction.

A standard computation shows that the $L_0$-characters of $\fz$ and $\Fun(\Op_{\ld{G}}(\cD))$ exist and are equal (by \cite{bd} 3.1.13-14, both sides have $\Aut$-equivariant filtrations 
and the associated graded pieces are isomorphic). Therefore, it suffices to see that the induced map:
\[
\vph^*:\Fun(\Op_{\ld{G}}(\cD))\lto \fz
\]
\noindent is injective.

\subsection{} Recall from \cite{fg2} that the Miura transform defines an injective map $\Fun(\Op_{\ld{G}}(\cD))\into\bW^{\fg[[t]]}$, where $\bW$ is the big Wakimoto module from the 
proof of Theorem \ref{t:ds}. The argument from \cite{wakimoto} Section 2 (which is a sort of ``birth of Miura opers" construction relying on 
Theorem \ref{t:main} and computations with the functor $\Psi$) implies that
the following diagram commutes:
\[\xymatrix{
& \Fun(\Op_{\ld{G}}(\cD))\ar[dr]\ar[dl]_{\vph^*} &   \\
\fz\ar@{=}[r] & \bV^{\fg[[t]]}\ar[r] & \bW^{\fg[[t]]}.
}\]
\noindent Since $\Fun(\Op_{\ld{G}}(\cD))\into\bW^{\fg[[t]]}$ is injective, we deduce the desired injectivity of $\vph^*$. This completes the proof of Theorem \ref{t:ff}.

\begin{rems}
\begin{enumerate}

\item There are two more standard compatibilities that the Feigin-Frenkel isomorphism satisfies. One is compatibility with the filtrations alluded to above. This is proved by exactly the same argument
above: both filtrations are compatible with the embeddings into $\bW^{\fg[[t]]}$. The second is compatibility with chiral Poisson structures, which is proved by a somewhat similar argument, as in \cite{wakimoto} 
Section 8.

\item The map $\Fun(\Op_{\ld{G}}(\cD))\into\bW^{\fg[[t]]}$ arises as follows. $\bW$ is the vacuum module of a chiral algebra, and $\bW^{\fg[[t]]}$ is the vacuum module for a \emph{commutative}
chiral algebra. In particular, $\bW^{\fg[[t]]}$ carries a canonical commutative algebra structure. Its spectrum naturally identifies with ``generic Miura opers," the scheme parametrizing opers on $\cD$ with
a second reduction to $\ld{B}$ at $0\in\cD$ such that the two reductions to $\ld{B}$ are transversal (i.e., their relative position is the longest element of the Weyl group).
In particular, this is an affine scheme equipped with a smooth map to $\Op_{\ld{G}}(\cD)$ (the fibers are the open Bruhat cell of $\ld{G}/\ld{B}$) and therefore defines an injective
map on the level of functions.

\end{enumerate}
\end{rems}

\section{Proof of the Finiteness Theorem}\label{s:finiteness}

\subsection{} In this section, we give a proof of the Finiteness Theorem.

We assume the formalism of higher category theory as developed by Lurie in \cite{htt} and \cite{higheralgebra}. Following Lurie, we refer to
$(\infty,1)$-categories merely as $\infty$-categories. When using language from category theory, we always understand it in the ``higher" sense,
so e.g. ``colimit" means ``homotopy colimit."

\subsection{} Let $C^+(\fgh\modules_{reg})^{G(O)}$ be the DG category form of the bounded below derived category of $G(O)$-equivariant, regular Kac-Moody modules, 
considered as a stable $\infty$-category via
the Dold-Kan correspondence. This DG category has a non-degenerate $t$-structure with heart $\fgh\modules_{reg}^{G(O)}$.

\subsection{} Note that the $t$-structure on $C^+(\fgh\modules_{reg})^{G(O)}$ is compatible with filtered colimits. Indeed, the forgetful functor to $C^+(\fgh\modules_{reg})$
is conservative, $t$-exact and commutes with colimits. Since the $t$-structure on $C^+(\fgh\modules_{reg})$ is compatible with filtered colimits by
\cite{higheralgebra} 1.3.4.21, we deduce the desired result.

\subsection{} Let $\sC$ be the (non-stable) $\infty$-category $C^{\geq{0}}(\fgh\modules_{reg})^{G(O)}$ of objects of $C^+(\fgh\modules_{reg})^{G(O)}$ with non-zero cohomologies only in non-negative 
(cohomological) degrees. We have adjoint functors:
\[
\Adjoint{\tau^{\geq{0}}}{C^+(\fgh\modules_{reg})^{G(O)}}{\sC}{i}
\]
\noindent where $i$ is the natural (fully faithful) inclusion functor. 

By the above, $\sC$ admits filtered colimits and $i$ commutes with filtered colimits. Moreover, $\sC$ admits finite colimits: these are computed by forming the corresponding colimit
in the stable $\infty$-category $C^+(\fgh\modules_{reg})^{G(O)}$ and then applying $\tau^{\geq 0}$. Therefore, $\sC$ admits all colimits.

\subsection{} We consider $\bV$ as an object of $C^+(\fgh\modules_{reg})^{G(O)}$ concentrated in cohomological degree $0$. Recall
from \cite{fg2} 7.5.1 that $\bV$ is \emph{almost compact} in $C^+(\fgh\modules_{reg})^{G(O)}$, i.e., the $\Ext$s out of $\bV$ commute with colimits bounded uniformly
from below.\footnote{Actually, what \emph{loc. cit.} shows is the corresponding result for $C^+(\fgh\modules_{reg})$, but the same proof goes through in the equivariant setting.}
In other words, $\bV[-n]$ is compact in $\sC$ for all $n\geq{0}$.

\begin{rem}
The failure of $\bV$ to be \emph{compact} (in the unbounded derived category) instead of merely being almost compact accounts for the technical difficulties
in this section. In particular, we could avoid working with non-stable $\infty$-categories if this were the case.
\end{rem}

\subsection{} We will use the following general lemma:

\begin{lem}\label{l:ind}
Let $\sC$ be a cocomplete\footnote{Note that we really mean presentable.} $\infty$-category and let $S$ be a set of compact objects of $\sC$. Suppose that $S$ detects equivalences in $\sC$ in the sense
that a map $f:X\lto Y$ in $\sC$ is an equivalence if and only if the induced maps $\Hom_{\sC}(s,X)\rar{f_*}\Hom_{\sC}(s,Y)$ are weak equivalences for all
$s\in{S}$. Let $\sD$ be the full subcategory of $\sC$ generated by $S$ under finite colimits, i.e., the smallest full subcategory of $\sC$ admitting finite colimits and containing $S$.
Then the induced functor $\Ind(\sD)\lto\sC$ is an equivalence.
\end{lem}

\begin{proof}

The functor $F:\Ind(\sD)\lto\sC$ is fully faithful since any object in $\sD$ is compact in $\sC$. Moreover, $F$ commutes with all colimits since the functor $\sD\lto\sC$ commutes with finite colimits. 
Therefore, $F$ admits a right adjoint $G$. Because $S$ detects equivalences, $G$ is conservative. Therefore, $F$ is an equivalence as desired.

\end{proof}

\subsection{} We apply Lemma \ref{l:ind} with the $\infty$-category $\sC$ given by $C^{\geq{0}}(\fgh\modules_{reg})^{G(O)}$ as above and with the set $S:=\{\bV[-n]\}_{n\geq{0}}$.
We have already noted that $S$ consists of objects compact in $\sC$. By Proposition \ref{p:generator}, the set $S$ detects equivalences in $\sC$. Therefore, with $\sD$ the full
subcategory of $\sC$ generated by shifts of the vacuum module under finite colimits, we deduce the following:

\begin{prop}\label{p:ind}
The natural functor $\Ind(\sD)\lto\sC$ is an equivalence.
\end{prop}

\subsection{} Let $\Rep_{fd}(\ld{G})$ be the tensor category of finite dimensional representations of $\ld{G}$. By \cite{bd} Section 7, the appendix to \cite{fg2}, and geometric Satake, 
there is a monoidal action of $\Rep_{fd}(\ld{G})$ on $C^+(\fgh\modules)^{G(O)}$. Moreover,
for any $V\in\Rep_{fd}(\ld{V})$, the corresponding functor:
\[
V*-:C^+(\fgh\modules_{reg})^{G(O)}\lto C^+(\fgh\modules_{reg})^{G(O)}
\]
\noindent has bounded cohomological amplitude (since it is given by de Rham cohomology on a finite-dimensional stratum of the affine Grassmannian). Note that the functor
$V*-$ admits a left and right adjoint given by convolution $V^{\vee}*-$ with the dual representation.

\begin{prop}\label{p:ac}
For $V$ in $\Rep_{fd}(\ld{G})$, the functor $V*-:C^+(\fgh\modules_{reg})^{G(O)}\lto C^+(\fgh\modules_{reg})^{G(O)}$ preserves almost compact objects. In particular, $V*\bV$ is almost compact in $C^+(\fgh\modules_{reg})^{G(O)}$. 
\end{prop}

\begin{proof}

Indeed, $V*-$ admits a right adjoint of finite cohomological amplitude which commutes with colimits. It is immediate to see that any such functor preserves almost compact objects.

\end{proof}

\subsection{} Let $C(\fz\modules)$ denote the DG category form of the derived category of $\fz\modules$ considered as a stable $\infty$-category. Recall that compact objects
of $C(\fz\modules)$ are the same as perfect objects.

\begin{lem}\label{l:trun}
An object of $C(\fz\modules)$ is compact if and only if it is bounded and has finitely presented cohomologies.
\end{lem}

\begin{proof}

This follows from coherence and (non-Noetherian) regularity (in the sense of Lemma \ref{l:res}) of $\fz$. Indeed, by coherence, perfect objects of $C(\fz\modules)$ have finitely presented cohomologies.
By regularity, any finitely presented $\fz$-module has a bounded resolution by free modules, so by d\'evissage we deduce the opposite inclusion.

\end{proof}

Let $C^{\geq{0}}(\fz\modules)$ be the full subcategory of $C(\fz\modules)$ consisting of objects with vanishing cohomologies in negative degrees. 
We deduce the following:

\begin{cor}\label{c:trun}
An object of $C^{\geq{0}}(\fz\modules)$ is compact in $C^{\geq{0}}(\fz\modules)$ if and only if it is compact in $C(\fz\modules)$.
\end{cor}

\begin{proof}

By ``cohomologies" of objects of $C^{\geq{0}}(\fz\modules)$, we understand cohomologies when considered as objects of
$C(\fz\modules)$.

Since the $t$-structure on $C(\fz\modules)$ is right complete by \cite{higheralgebra} 1.3.4.21, any compact object in $C^{\geq{0}}(\fz\modules)$ is cohomologically bounded above. 
Note that top cohomology of any compact object of $C^{\geq{0}}(\fz\modules)$ is a finitely presented $\fz\modules$ since the $t$-structure on $C(\fz\modules)$ is compatible with filtered colimits.
Indeed, this top cohomology is given by applying a truncation functor whose right adjoint commutes with filtered colimits, and therefore it preserves compact objects.
But by d\'evissage, the result now follows immediately from Lemma \ref{l:trun}.

\end{proof}

\subsection{} Let $\Psi^{\geq 0}:\sC\lto C^{\geq{0}}(\fz\modules)$ be the functor $\tau^{\geq{0}}\circ\Psi\circ i$ (where $\tau^{\geq{0}}$ denotes the truncation functor for $C(\fz\modules)$).
Because $\Psi$ is $t$-exact, $\Psi^{\geq 0}$ is also computed as $\Psi\circ i$. Since $\Psi$ commutes with filtered colimits and since the $t$-structures are compatible with filtered
colimits, $\Psi^{\geq{0}}$ also commutes with filtered colimits. Moreover, by the description of how to compute colimits in $\sC$ and $C^{\geq{0}}(\fz\modules)$, we see
that $\Psi^{\geq{0}}$ actually commutes with all colimits.

\subsection{} We have the following proposition:

\begin{prop}\label{p:perfect}
For any $X\in\sC$ compact in $\sC$, $\Psi^{\geq{0}}(X)\in C^{\geq 0}(\fz\modules)$ is compact.
\end{prop}

\begin{proof}

By \cite{htt} 5.4.2.4 and Proposition \ref{p:ind}, the subcategory of compact objects of $\sC$ is generated by $\{\bV[-n]\}_{n\geq{0}}$ by finite colimits
and retracts. Since $\Psi^{\geq{0}}$ commutes with colimits and since each of these operations preserves compactness, it suffices to see that $\Psi^{\geq{0}}$ maps each $\bV[-n]$ to an object 
compact in $C^{\geq{0}}(\fz\modules)$. But this follows immediately from Theorem \ref{t:ds}.

\end{proof}

\subsection{} Now let us deduce the Finiteness Theorem from the above results. Let $V\in\Rep_{fd}(\ld{G})$ be a finite-dimensional representation of $\ld{G}$.
By Proposition \ref{p:ac} and Proposition \ref{p:perfect}, $\Psi^{\geq{0}}(V*\bV)$ is compact in $C^{\geq{0}}(\fz\modules)$. By Corollary
\ref{c:trun}, we deduce that $\Psi^{\geq{0}}(V*\bV)$ is perfect when considered as an object of $C(\fz\modules)$. However, as noted above,
by $t$-exactness of $\Psi$, we have $\Psi^{\geq{0}}(V*\bV)=\Psi(V*\bV)$.

\end{document}